\documentclass[preprint,3p]{elsarticle}

\usepackage{eurosym}
\usepackage{amsfonts}
\usepackage{amsmath}
\usepackage{amsthm}
\usepackage{amssymb}
\usepackage{mathrsfs}
\usepackage{xcolor} 
\usepackage[colorlinks=true,
linkcolor=magenta,
citecolor=cyan,
  urlcolor=orange]{hyperref}
\usepackage{graphicx}
\usepackage[tableposition=top]{caption}
\usepackage{subcaption}
\usepackage{float}
\usepackage{color}
\usepackage{setspace}
\usepackage{placeins}
\usepackage{mathrsfs}
\usepackage{enumitem}
\usepackage{empheq}
\usepackage[utf8]{inputenc}
\newsavebox \foobox
\newlength{\foodim}

\graphicspath{{C:/Figures/}}
\newtheorem{theorem}{Theorem}

\newtheorem{definition}{Definition}

\newtheorem{lemma}{Lemma}

\newtheorem{proposition}{Proposition}
\newtheorem{remark}{Remark}

\numberwithin{equation}{section}

\journal{Journal of Differential Equations}

\begin{document}
\begin{frontmatter}
\title{On the global existence and uniform-in-time bounds for three-component reaction-diffusion systems with mass control and polynomial growth}
\author{Redouane Douaifia$^{a,b}$, Salem Abdelmalek$^{c}$, Mokhtar Kirane$^{d,*}$}
\address{(a) Process Engineering Department, Faculty of Technology, University of Blida 1, Blida, Algeria
\\
(b) Water Environment and Sustainable Development Laboratory, Faculty of Technology, Blida 1 University, PO Box 270-09000, Blida, Algeria
\\
(c) Department of Mathematics, Laboratory (LAMIS), Echahid Cheikh Larbi Tebessi University, Tebessa, Algeria
\\
(d) Department of Mathematics, College of Computing and Mathematical Sciences, Khalifa University, P.O. Box: 127788, Abu
Dhabi, UAE
\\
(*) Corresponding author, email: mokhtar.kirane@ku.ac.ae
}
\begin{abstract}
We investigate a class of three-component reaction–diffusion systems subject to mass control and a newly introduced structural assumption, referred to as linear intermediate weighted sum condition. Under these hypotheses, we establish the global existence of classical solutions in arbitrary spatial dimensions and wide class of boundary conditions, even when the nonlinearities exhibit arbitrary polynomial growth. We establish also that, under slight-stronger assumptions and mixed boundary conditions, solutions admit uniform-in-time bounds. Our approach relies on the extension of $L^p$-energy polynomial functionals, together with the regularizing effect for parabolic equations. Furthermore, we demonstrate the applicability of our framework by analyzing three-species sub-skew-symmetric Lotka–Volterra systems with higher-order interactions.
\end{abstract}
\end{frontmatter}
\tableofcontents
\section{Introduction}
This work examines the global existence of classical solutions for a
class of semilinear reaction-diffusion systems consisting of three
equations. 
\subsection{Problem setting}
Let $N\in\mathbb{N}$, and $\Omega \subset\mathbb{R}^{N}$ is a bounded domain with smooth boundary $\partial \Omega $. We consider
the reaction-diffusion system.

\begin{equation}
\left\{ 
\begin{array}{cc}
\displaystyle\frac{\partial u}{\partial t}-d_{1}\Delta u=f\left( u,v,w\right) , & \text{%
in }\Omega \times \left( 0,+\infty \right) , \\ \\
\displaystyle\frac{\partial v}{\partial t}-d_{2}\Delta v=g\left( u,v,w\right) , & \text{%
in }\Omega \times \left( 0,+\infty \right) , \\ \\
\displaystyle\frac{\partial w}{\partial t}-d_{3}\Delta w=h\left( u,v,w\right) , & \text{%
in }\Omega \times \left( 0,+\infty \right) , \\ \\
\displaystyle \frac{\partial u}{\partial \eta }\left( x,t\right) =\frac{\partial v}{%
\partial \eta }\left( x,t\right) =\frac{\partial w}{\partial \eta }\left(
x,t\right) =0, & \text{on }\partial \Omega \times \left( 0,+\infty \right) ,
\\ \\
u\left( x,0\right) =u_{0}(x)\text{, }v\left( x,0\right) =v_{0}(x)\text{, }%
w\left( x,0\right) =w_{0}(x), & \text{in }\Omega ,%
\end{array}%
\right.  \label{main_system}
\end{equation}
where $d_{1},d_{2}>0$ are diffusion coefficients, $\eta $ is the unit
outward normal vector on $\partial \Omega $, $u_{0}$, $v_{0}$, and $w_{0}$
are initial data. Under the notations \(\xi = (\xi_1, \xi_2, \xi_3)\) and \(\Lambda(\xi) = 1 + \xi_1 + \xi_2 + \xi_3\), the following assumptions are made regarding the nonlinearities:
\begin{enumerate}[label=($\mathcal{A}${{\arabic*}})]
\item \label{a1}(Quasi-positivity and Local Lipschitz) The nonlinearities $%
f,g,h:%
\mathbb{R}
^{3}\longrightarrow 
\mathbb{R}
$ are locally Lipschitz, and 
\begin{equation}
f\left( 0,\xi _{2},\xi _{3}\right) ,g\left( \xi _{1},0,\xi _{3}\right)
,h\left( \xi _{1},\xi _{2},0\right) \geq 0,\ \text{ for any } \xi \in \mathbb{R}_{+}^{3}.
\end{equation}

\item \label{a2}(Mass Control) There exists $K_1\geq 0$, such that 
\begin{equation}
f\left( \xi\right) +g\left( \xi\right) +h\left( \xi\right) \leq K_1 \Lambda(\xi),\ \text{ for any } \xi \in 
\mathbb{R}_{+}^{3}.
\end{equation}

\item \label{a3}(Linear Intermediate Weighted Sum) There exist $K_{i}\geq 0$
for $i=2,3,4$, and $\lambda _{1},\lambda _{2}>1$ (or $\lambda _{1},\lambda
_{2}<1$), such that 
\begin{equation}
\left\{ 
\begin{array}{c}
\lambda _{1}f\left( \xi\right) +g\left( \xi\right) +h\left( \xi\right) \leq
K_{2}\Lambda(\xi), \\ \\
\lambda _{2}f\left( \xi\right) +\lambda _{2}g\left(
\xi \right) +h\left( \xi\right)
\leq K_{3}\Lambda(\xi), \\ \\
\lambda _{1}\lambda _{2}f\left( \xi\right) +\lambda
_{2}g\left( \xi\right) +h\left( \xi\right) \leq K_{4}\Lambda(\xi),%
\end{array}%
\right.\label{EqCond_a3}
\end{equation}%
for any $ \xi \in \mathbb{R}_{+}^{3}.$

\item \label{a4} (Polynomial Growth) There exist $m\in \mathbb{N}$ and $M>0$%
, such that 
\begin{equation}
f\left( \xi\right) ,g\left( \xi\right) ,h\left( \xi\right) \leq M(\Lambda(\xi))^{m},\ \text{for any } \xi \in 
\mathbb{R}_{+}^{3}.
\end{equation}
\end{enumerate}

We now turn to a discussion of the structural assumptions \ref{a1}-\ref{a4}. Assumption \ref{a1}, which requires quasi-positivity together with local Lipschitz continuity, is classical in the analysis of reaction–diffusion systems. These conditions guarantee local existence of solutions and, importantly, ensure their non-negativity, i.e. the invariance of the positive octant under the flow. The latter property is particularly relevant for systems arising in biology or chemistry, since it preserves the physical meaning of the model: if the initial concentration (or density, or population) is non-negative, then the corresponding solution remains non-negative throughout its evolution. Assumption \ref{a2} serves as a unifying framework that extends several classical hypotheses commonly used in the literature, in particular:
\begin{itemize}
    \item mass conservation:
    \begin{equation}
    \label{mass_conservation}
f\left( \xi\right) +g\left( \xi\right) +h\left( \xi\right) =0,\ \text{ for any } \xi \in 
\mathbb{R}_{+}^{3}.
\end{equation}
\item mass dissipation:
\begin{equation}
\label{mass_dissipation}
f\left( \xi\right) +g\left( \xi\right) +h\left( \xi\right) \leq 0,\ \text{ for any } \xi \in 
\mathbb{R}_{+}^{3}.
\end{equation}
\end{itemize}
The assumption \ref{a3} that we have called "\textit{intermediate wheighted sum condition} (IWSC)", which is a modefied version of the so-called "intermediate sum condition (ISC)":\\
There exist an $3\times 3$ lower triangle matrix $A = (a_{ij})$ with nonnegative elements and $a_{ii}>0$ for all $i=1,2, 3$, a constant $r\geq 1$ and a constant $K>0$, such that
\begin{equation}
\left\{ 
\begin{array}{c}
a_{11}f\left( \xi\right)  \leq
K\Lambda^r(\xi), \\ \\
a_{21}f\left( \xi\right) +a _{22}g\left(
\xi \right) 
\leq K\Lambda^r(\xi), \\ \\
a_{31}f\left( \xi\right) +a_{32}g\left( \xi\right) +a_{33}h\left( \xi\right) \leq K\Lambda^r(\xi),%
\end{array}%
\right.\label{EqCond_a3_Morgan}
\end{equation}%
for any $ \xi \in \mathbb{R}_{+}^{3}.$
The essential difference between these conditions is that, in (ISC) one of the nonlinearties is assumed to be upper bounded by a polynomial of order $r$ (in particular less than $1+\frac{2}{N}$, or, for small dimension, $r$ is at most $3$), while in (IWSC), we need three wheighted sum of all nonlinearties.
Assumption \ref{a4} requires that the nonlinearities exhibit at most polynomial growth of order $m$. It is worth emphasizing that the results established in this paper impose no restriction on the value of $m$.

\subsection{State of the art and motivation}
The investigation of global existence for reaction-diffusion systems has attracted continuous attention for last decades as a fundamental question. Despite being a classical subject, it remains a source of numerous open questions and mathematical challenges, even in systems that look simple and are based on real-world uses, such reversible chemical processes. A major difficulty arises from the fact that, reaction-diffusion systems generally (consisting of two or more equations) lack a maximum principle or invariant regions, often nonlinearities have not a constant sign and this means that the local solution is a priori bounded or at least bounded in some $L^p$-space. Early contributions to this topic include the works of Alikakos \cite{Alikakos1979}, Haraux and Kirane \cite{Kirane1983}, Masuda \cite{Masuda1983}, Rothe \cite{Rothe1984}, Amann \cite{Amann1985}, Hollis, Martin, and Pierre \cite{Hollis1987}, Haraux and Youkana \cite{Youkana1988}, Morgan \cite{Morgan1989}, Kouachi \cite{Kouachi2001,Kouachi2002}, as well as Abdelmalek and Kouachi \cite{Abdelmalek2007}.

It is well known that assumptions \ref{a1} and \ref{a2} by themselves do not exclude the possibility of finite-time blow-up of solutions to system \eqref{main_system}; see \cite{Pierre2000}. This observation naturally led to an intensive study of systems supplemented by additional structural conditions, in particular \ref{a4}. Some advance were obtained by \cite{Goudon2010,Caputo2009,Canizo2014,PierreSuzuki2019,MorganTang2020}, where, under the entropy inequality or without, global existence of bounded solutions was established with specific assumptions, for instance, small space dimensions ($N =1, 2$) and at most cubic systems or strictly sub-quadratic systems in all dimensions.
The global existence of quadratic systems in higher dimensions remained unresolved until a series of contributions \cite{Souplet2018,CaputoGoudon2019,FellnerMorganTang2020} closed this gap. In the first two, the entropy inequality continued to play a role, whereas the last one succeeded in removing this requirement. Moreover, Fellner \textit{et al.}, complemented the analysis in \cite{FellnerMorganTang2021} by establishing the uniform boundedness in time of solutions.
It is worth noting that most of the previous works dealt with systems that impose homogeneous conditions (Neumann and/or Dirichlet), but some of them dealt with non-homogeneous boundary conditions as in \cite{Morgan1989,Hollis1993,Kouachi2002,Abdelmalek2007}. For more details we refer the reader to the surveys \cite{Pierre2010,Quittner2019}.

The results summarized above indicate that global existence together with uniform-in-time bounds for nonnegative classical solutions can be ensured either in the presence of quadratic nonlinearities (or super-quadratic ones subject to dimension constraints), or under structural conditions of the form \eqref{EqCond_a3_Morgan}. However, such conditions are not generally satisfied by a wide class of systems. To illustrate this limitation, we consider the following nonlinearities:
\begin{equation}
\left\{
\begin{aligned}
f(\xi) &= \xi_2^{5}-\xi_1^{6}, \\
g(\xi) &= \xi_3^{7}-B\xi_2^{5}, \\
h(\xi) &= \xi_1^{6}-C\xi_3^{7},
\end{aligned}
\right.
\label{ReacFuncs_Intro}
\end{equation}
where $B$ and $C$ are positive constants chosen sufficiently large. It is clear that the nonlinearities \eqref{ReacFuncs_Intro} do not satisfy the so-called \textit{linear intermediate sum condition} \eqref{EqCond_a3_Morgan}, as well as are not quadratic. Nevertheless, it can be verified that the nonlinearities \eqref{ReacFuncs_Intro} satisfy our proposed condition, referred to as \textit{the linear intermediate weighted sum} \eqref{EqCond_a3}, by an appropriate choice of the parameters $B$ and $C$. This observation constitutes a central motivation of the present work. Specifically, by employing the linear intermediate weighted sum condition \ref{a3}, we are able to accommodate nonlinearities with arbitrary polynomial growth. A natural question then arises: can one establish uniform-in-time bounds for solutions of \eqref{main_system} subject to mixed boundary conditions under assumptions \ref{a1}-\ref{a4}? We give a positive answer to this issue in this work.

\subsection{Main results and key ideas}

During this section, we will use the following notation: $A_{ij}=\frac{d_{i}+d_{j}}{2\sqrt{d_{i}d_{j}}}$ for $i,j=1,2,3$. The main contribution of this paper is encapsulated in the following
theorems.

\begin{theorem}
\label{main_Theor} Assume the conditions \ref{a1}, \ref{a2}, \ref{a3}, \ref{a4}, for some $\lambda _{1}\geq $ $\theta ^{2(p-1)}$, $\lambda _{2}\geq $ $\sigma
^{2(p-1)}$, with
\begin{equation}
\left\{ 
\begin{array}{c}
\theta >A_{12}, \\ 
\text{ and }\\
\left( \theta ^{2}-\left( A_{12}\right) ^{2}\right) \left( \theta ^{2}\sigma
^{2}-\left( A_{13}\right) ^{2}\right) >\left( A_{23}\theta ^{2}-A_{12}A_{13}%
\right) ^{2}, \\ 
\end{array}%
\right.   \label{Teta_Sigma_Conds}
\end{equation}
such that
\begin{equation}
p=\min \left\{ l\in \mathbb{N}\ :\ l>\frac{m(N+2)}{2}\right\} ,  \label{Cond_p}
\end{equation}%
and $u_{0},v_{0},w_{0}\in L^{\infty }\left( \Omega ;\mathbb{R}_{+}\right) $.
Then the system (\ref{main_system}) possesses a unique nonnegative global
classical solution. Moreover, if $K_i= 0$ for all $i=1,\dots,4$, then the solution is bounded uniformly in time, i.e. 
\begin{equation}\label{uvw_LinftyUniformUpperBound}
		\sup_{t\geq 0}\|u(.,t)\|_{L^{\infty}(\Omega)},\quad\sup_{t\geq 0}\|v(.,t)\|_{L^{\infty}(\Omega)},\quad \sup_{t\geq 0}\|w(.,t)\|_{L^{\infty}(\Omega)} < +\infty.
\end{equation}
\end{theorem}

We have also the following result.
\begin{theorem}
\label{main_Theor2} Assume the conditions \ref{a1}, \ref{a2}, \ref{a3} for
some $\lambda _{1}\leq $ $\theta ^{-2p}$, $\lambda _{2}\leq $ $\sigma ^{-2p}$%
, and \ref{a4} where the assumptions (\ref{Teta_Sigma_Conds}), (\ref{Cond_p}) are fulfilled, and $u_{0},v_{0},w_{0}\in
L^{\infty }\left( \Omega ;\mathbb{R}_{+}\right) $. Then the system (\ref%
{main_system}) possesses a unique nonnegative global classical solution. Moreover, if $K_i= 0$ for all $i=1,\dots,4$, then the solution is bounded uniformly in time, as \eqref{uvw_LinftyUniformUpperBound}.
\end{theorem}

\begin{remark}[Generalization]
\label{Gen_Theorems}
We can consider the following system with general boundary conditions.
\begin{equation}
\left\{ 
\begin{array}{cc}
\displaystyle\frac{\partial u}{\partial t}-d_{1}\Delta u=f\left( u,v,w\right) , & \text{%
in }\Omega \times \left( 0,+\infty \right) , \\ \\
\displaystyle\frac{\partial v}{\partial t}-d_{2}\Delta v=g\left( u,v,w\right) , & \text{%
in }\Omega \times \left( 0,+\infty \right) , \\ \\
\displaystyle\frac{\partial w}{\partial t}-d_{3}\Delta w=h\left( u,v,w\right) , & \text{%
in }\Omega \times \left( 0,+\infty \right) , \\ \\
\begin{array}{c}
\displaystyle\lambda _{1}u(x,t)+\left( 1-\lambda _{1}\right) \frac{\partial u}{\partial
\eta }\left( x,t\right) =\beta _{1}, \\ \\
\displaystyle\lambda _{2}v(x,t)+\left( 1-\lambda _{2}\right) \frac{\partial v}{\partial
\eta }\left( x,t\right) =\beta _{2}, \\ \\
\displaystyle\lambda _{3}w(x,t)+\left( 1-\lambda _{3}\right) \frac{\partial w}{\partial
\eta }\left( x,t\right) =\beta _{3},%
\end{array}
& \text{on }\partial \Omega \times \left( 0,+\infty \right), \\ \\
u\left( x,0\right) =u_{0}(x)\text{, }v\left( x,0\right) =v_{0}(x)\text{, }%
w\left( x,0\right) =w_{0}(x), & \text{in }\Omega.
\end{array}%
\right.  \label{main_system_GBCs}
\end{equation}
The same result (for global existence) as in Theorem \ref{main_Theor} and Theorem \ref{main_Theor2} can be obtained for homogeneous
Dirichlet BCs, nonhomogeneous Robin BCs, and a mixture of homogeneous
Dirichlet with nonhomogeneous Robin boundary conditions, i.e., the boundary conditions are
defined as follows:

\begin{itemize}
\item For homogeneous Neumann boundary conditions, we use%
\begin{equation}
\lambda _{i}=0\quad \text{and} \quad\beta _{i}\equiv 0,\quad \text{for}\quad i=1,2,3.  
\label{HNBCs_EX3}
\end{equation}

\item For homogeneous Dirichlet boundary conditions, we use 
\begin{equation}
1-\lambda _{i} =0 \quad \text{and} \quad \beta _{i}\equiv 0,\quad \text{for} \quad i=1,2,3. 
\label{HDBCs_EX3}
\end{equation}

\item For nonhomogeneous Robin boundary conditions, we use 
\begin{equation}
\lambda _{i}\in \left( 0,1\right) \text{, and }\beta _{i}\in C^{1}\left(
\partial \Omega ;%
\mathbb{R}
_{+}\right),\quad \text{for}\quad i=1,2,3.
\label{NHRBCs_EX3}
\end{equation}

\item For a mixture of homogeneous Dirichlet with nonhomogeneous Robin
boundary conditions, we use 
\begin{equation}
\left\{
\begin{array}{ll}
\exists\, i\in\{1,2,3\}: & \lambda_{i}=1,\quad \text{and} \quad \beta_{i}\equiv 0, \\ \\
\text{when} \quad j\neq i,\quad\text{for}\quad j=1,2,3: 
& \lambda_{j}\in(0,1),\quad \text{and}\quad \beta_{j}\in C^{1}\!\left(\partial\Omega;\mathbb{R}_{+}\right).
\end{array}
\right.
\label{MixNHRBCsHDBCs_EX3}
\end{equation}
\end{itemize}
\end{remark}
We now state the results that generalize Theorem \ref{main_Theor} and Theorem \ref{main_Theor2}.
\begin{theorem}
\label{main_Theor3GBCs}
Assume the conditions \ref{a1}, \ref{a2}, \ref{a3} for some $\lambda _{1}\geq $ $\theta ^{2(p-1)}$, $\lambda _{2}\geq $ $\sigma
^{2(p-1)}$, and \ref{a4}  where the assumptions (\ref{Teta_Sigma_Conds}) and (\ref{Cond_p}) are fulfilled.
Then, for any $u_{0},v_{0},w_{0}\in L^{\infty }\left( \Omega ;\mathbb{R}_{+}\right) $ the system (\ref{main_system_GBCs}) with one of the boundary conditions $\left( \text{\ref{HNBCs_EX3}}\right) $-$\left( \text{\ref{MixNHRBCsHDBCs_EX3}}\right)$ imposed, possesses a unique nonnegative global
classical solution. Moreover, if $K_i= 0$ for all $i=1,\dots,4$, and in \eqref{NHRBCs_EX3} or \eqref{MixNHRBCsHDBCs_EX3}, we have $\beta_i\equiv 0$ for $i=1,2,3$, then the solution is bounded uniformly in time, i.e. 
\begin{equation}\label{uvw_LinftyUniformUpperBound_GBCs}
		\sup_{t\geq 0}\|u(.,t)\|_{L^{\infty}(\Omega)},\quad\sup_{t\geq 0}\|v(.,t)\|_{L^{\infty}(\Omega)},\quad \sup_{t\geq 0}\|w(.,t)\|_{L^{\infty}(\Omega)} < +\infty.
\end{equation}
\end{theorem}
Furthermore, we have the following result.
\begin{theorem}
\label{main_Theor4GBCs}
Assume the conditions \ref{a1}, \ref{a2}, \ref{a3} for
some $\lambda _{1}\leq $ $\theta ^{-2p}$, $\lambda _{2}\leq $ $\sigma ^{-2p}$%
, and \ref{a4} where the assumptions (\ref{Teta_Sigma_Conds}) and (\ref{Cond_p}) are fulfilled.
Then, for any $u_{0},v_{0},w_{0}\in L^{\infty }\left( \Omega ;\mathbb{R}_{+}\right) $ the system (\ref{main_system_GBCs}) with one of the boundary conditions $\left( \text{\ref{HNBCs_EX3}}\right) $-$\left( \text{\ref{MixNHRBCsHDBCs_EX3}}\right)$ imposed, possesses a unique nonnegative global
classical solution. Moreover, if $K_i= 0$ for all $i=1,\dots,4$, and in \eqref{NHRBCs_EX3} or \eqref{MixNHRBCsHDBCs_EX3}, we have $\beta_i\equiv 0$ for $i=1,2,3$, then the solution is bounded uniformly in time.
\end{theorem}


\subsection{Structure of the paper}
 The rest of this paper is organized as follows. Section \ref{Sec_ProofMRs} is devoted to the proofs of the main results; Subsection \ref{Sec2_Preliminaries} recalls some preliminaries, Subsection \ref{Thrm1_GE} establishes the global existence of solutions in Theorem \ref{main_Theor}, Subsection \ref{Thrm1_UBt} derives the uniform-in-time bounds associated with Theorem \ref{main_Theor}, Subsection \ref{Thrm2_Proof} provides the proof of Theorem \ref{main_Theor2}, and Subsection \ref{Sub_sec_GeneralizationsProof} presents the generalizations leading to the proofs of Theorems \ref{main_Theor3GBCs} and \ref{main_Theor4GBCs}. Section \ref{Sec_Applications} discusses applications of the results to specific models; Example~1 is treated in Subsection \ref{Sec_Applications_Ex1}, Example~2 in Subsection \ref{Sec_Applications_Ex2}, and Example~3 of three-species sub-skew-symmetric Lotka–Volterra systems with higher-order interactions, in Subsection \ref{Sec_Applications_Ex3}.

\section{Proof of main results}
\label{Sec_ProofMRs}
\subsection{Preliminaries}
\label{Sec2_Preliminaries}
We start with the definition of classical solutions.
\begin{definition}[Classical solutions]
	Let $0<T\leq\infty$. A classical solution to \eqref{main_system_GBCs} on $(0,T)$ is a vector of functions $S:=(u,v,w)$, satisfying $u,v,w\in  C([0,T];L^p(\Omega))\cap L^\infty(\Omega\times(0,T))\cap C^{1,2}(\overline{\Omega}\times(\tau,T))$ for all $p>1$ and all $0<\tau<T$, and $S$ satisfies each equation in \eqref{main_system_GBCs} pointwise in $\Omega\times(0,T)$.
\end{definition}
According to assumption \ref{a1}, it is a classical task to show
the existence of a unique local nonnegative classical solution of system (%
\ref{main_system}) on $[0,T_{max})$, where $T_{max}$ is the eventual
blow-up time in $L^{\infty }\left( \Omega \right) $ (see e.g. \cite%
{Henry1981,Rothe1984}).

\begin{proposition}
\label{Theor_LocExst} Suppose that $u_{0},v_{0},w_{0}\in L^{\infty }\left(
\Omega ;\mathbb{R}_{+}\right) $ and \ref{a1} holds. Then there exists $%
T_{max}>0$ such that (\ref{main_system}) admits a unique nonnegative
classical solution on $[0,T_{max})$. Moreover, if $T_{max}<+\infty $, then 
\begin{equation*}
\limsup_{t\rightarrow T_{max}^{-}}\left( \Vert u(\cdot ,t)\Vert _{L^{\infty
}(\Omega )}+\Vert v(\cdot ,t)\Vert _{L^{\infty }(\Omega )}+\Vert w(\cdot
,t)\Vert _{L^{\infty }(\Omega )}\right) =+\infty.
\end{equation*}
\end{proposition}
It is essential to first establish the following lemmas.

\begin{lemma}
\label{LIWS_lemma} Let $\Phi (\xi ),\Psi (\xi)$ be two real-valued functions defined on $\mathbb{R}_{+}^{3}$%
, $\alpha >1$ (or $0<\alpha <1$), and $C_{1},C_{2}\geq 0$. Then, for all $
\xi\in \mathbb{R}_{+}^{3}$, we have 
\begin{equation}
\text{ }\left\{ 
\begin{array}{c}
\Phi (\xi )+\Psi (\xi)\leq
C_{1}\Lambda(\xi), \\ \\
\text{and} \\ \\
\alpha \Phi (\xi )+\Psi (\xi)\leq C_{2}\Lambda(\xi),
\end{array}%
\right.
\end{equation}%
implies 
\begin{equation*}
\alpha ^{\ast }\Phi (\xi )+\Psi (\xi)\leq \max \left\{ C_{1},C_{2}\right\} \Lambda(\xi),
\end{equation*}
for all $\alpha^{\ast }\in \lbrack 1,\alpha ]$ \ (or for all 
$\alpha ^{\ast }\in \lbrack \alpha ,1]$).
\end{lemma}

\begin{proof}
The proof is by contradiction. Let $\xi\in 
\mathbb{R}_{+}^{3}$, and there exists $\alpha ^{\ast }\in \lbrack 1,\alpha ]$%
, such that 
\begin{equation*}
\left\{ 
\begin{array}{c}
\Phi (\xi )+\Psi (\xi )\leq
C_{1}\Lambda(\xi), \\ 
\text{and} \\ 
\alpha \Phi (\xi )+\Psi (\xi )\leq C_{2}\Lambda(\xi),%
\end{array}%
\right. 
\end{equation*}%
and%
\begin{equation*}
\alpha ^{\ast }\Phi (\xi )+\Psi (\xi )>\max \left\{ C_{1},C_{2}\right\} \Lambda(\xi).
\end{equation*}%
Thus 
\begin{equation*}
\left\{ 
\begin{array}{c}
\Phi (\xi )+\Psi (\xi)<\alpha
^{\ast }\Phi (\xi)+\Psi (\xi),
\\ 
\text{and} \\ 
\alpha \Phi (\xi )+\Psi (\xi )<\alpha ^{\ast }\Phi (\xi )+\Psi (\xi).
\end{array}%
\right. 
\end{equation*}%
We obtain%
\begin{equation*}
\left\{ 
\begin{array}{c}
\Phi (\xi )+\Psi (\xi )-\alpha
^{\ast }\Phi (\xi)-\Psi (\xi )<0,
\\ 
\text{and} \\ 
\alpha \Phi (\xi)+\Psi (\xi )-\alpha ^{\ast }\Phi (\xi )-\Psi (\xi)<0.
\end{array}%
\right. 
\end{equation*}%
Hence, for the case $\alpha ^{\ast }\in (1,\alpha )$, we get%
\begin{equation*}
\left\{ 
\begin{array}{c}
(1-\alpha ^{\ast })\Phi (\xi )<0 \text{ implies } \Phi(\xi )>0, \\ 
\text{and} \\ 
(\alpha -\alpha ^{\ast })\Phi (\xi)<0 \text{ implies }
\Phi (\xi )<0.%
\end{array}%
\right.
\end{equation*}%
This leads to a contradiction.
For the case $\alpha ^{\ast }\in (\alpha ,1)$, we obtain
\begin{equation*}
\left\{ 
\begin{array}{c}
(1-\alpha ^{\ast })\Phi (\xi)<0\text{ implies } \Phi
(\xi)<0, \\ 
\text{and} \\ 
(\alpha -\alpha ^{\ast })\Phi (\xi)<0\text{ implies }
\Phi (\xi )>0.%
\end{array}%
\right. 
\end{equation*}%
We arrive at a contradiction. Therefore, the desired result holds.
\end{proof}

\begin{lemma}
\label{LemmaD1Hp}
Let $p$ be a positive integer, and 
\begin{equation}
\widehat{\mathcal{H}}_{p}(\xi)=\sum_{j=0}^{p}\sum_{i=0}^{j}C_{p}^{j}C_{j}^{i}\theta _{i}\sigma
_{j}\xi _{1}^{i}\xi _{2}^{j-i}\xi _{3}^{p-j}\text{, for all } \xi\in 
\mathbb{R}
_{+}^{3},  \label{HomPolyH}
\end{equation}%
be a homogeneous polynomial with $\left( \theta _{i}\right) _{0\leq i\leq p}$
and $\left( \sigma _{j}\right) _{0\leq j\leq p}$ arbitrary finite
sequences. Then, we have 
\begin{equation*}
\partial _{\xi _{1}}\widehat{\mathcal{H}}_{p}(\xi)=p\sum_{j=0}^{p-1}\sum_{i=0}^{j}C_{p-1}^{j}C_{j}^{i}\theta _{i+1}\sigma
_{j+1}\xi _{1}^{i}\xi _{2}^{j-i}\xi _{3}^{\left( p-1\right) -j},
\end{equation*}%
\begin{equation*}
\partial _{\xi _{2}}\widehat{\mathcal{H}}_{p}(\xi )=p\sum_{j=0}^{p-1}\sum_{i=0}^{j}C_{p-1}^{j}C_{j}^{i}\theta _{i}\sigma
_{j+1}\xi _{1}^{i}\xi _{2}^{j-i}\xi _{3}^{\left( p-1\right) -j},
\end{equation*}%
\begin{equation*}
\partial _{\xi _{3}}\widehat{\mathcal{H}}_{p}(\xi )=p\sum_{j=0}^{p-1}\sum_{i=0}^{j}C_{p-1}^{j}C_{j}^{i}\theta _{i}\sigma
_{j}\xi _{1}^{i}\xi _{2}^{j-i}\xi _{3}^{\left( p-1\right) -j}.
\end{equation*}
\end{lemma}

\begin{proof}
It is sufficient to follow the proof steps provided in \cite{Kouachi2002,Abdelmalek2007}.
\end{proof}

\begin{lemma}
\label{LemmaD2Hp}
Let $\widehat{\mathcal{H}}_{p}$ be the homogeneous polynomial defined in $%
\left( \text{\ref{HomPolyH}}\right) $. Then%
\begin{equation*}
\partial^{2} _{\xi _{1}^{2}}\widehat{\mathcal{H}}_{p}(\xi )=p(p-1)\sum_{j=0}^{p-2}\sum_{i=0}^{j}C_{p-2}^{j}C_{j}^{i}\theta
_{i+2}\sigma _{j+2}\xi _{1}^{i}\xi _{2}^{j-i}\xi _{3}^{\left( p-2\right) -j},
\end{equation*}%
\begin{equation*}
\partial^{2} _{\xi _{2}^{2}}\widehat{\mathcal{H}}_{p}(\xi)=p(p-1)\sum_{j=0}^{p-2}\sum_{i=0}^{j}C_{p-2}^{j}C_{j}^{i}\theta
_{i}\sigma _{j+2}\xi _{1}^{i}\xi _{2}^{j-i}\xi _{3}^{\left( p-2\right) -j},
\end{equation*}%
\begin{equation*}
\partial^{2} _{\xi _{3}^{2}}\widehat{\mathcal{H}}_{p}(\xi)=p(p-1)\sum_{j=0}^{p-2}\sum_{i=0}^{j}C_{p-2}^{j}C_{j}^{i}\theta
_{i}\sigma _{j}\xi _{1}^{i}\xi _{2}^{j-i}\xi _{3}^{\left( p-2\right) -j},
\end{equation*}%
\begin{equation*}
\partial^{2} _{\xi _{1}\xi _{2}}\widehat{\mathcal{H}}_{p}(\xi)=p(p-1)\sum_{j=0}^{p-2}\sum_{i=0}^{j}C_{p-2}^{j}C_{j}^{i}\theta
_{i+1}\sigma _{j+2}\xi _{1}^{i}\xi _{2}^{j-i}\xi _{3}^{\left( p-2\right) -j},
\end{equation*}%
\begin{equation*}
\partial^{2} _{\xi _{1}\xi _{3}}\widehat{\mathcal{H}}_{p}(\xi)=p(p-1)\sum_{j=0}^{p-2}\sum_{i=0}^{j}C_{p-2}^{j}C_{j}^{i}\theta
_{i+1}\sigma _{j+1}\xi _{1}^{i}\xi _{2}^{j-i}\xi _{3}^{\left( p-2\right) -j},
\end{equation*}%
\begin{equation*}
\partial^{2} _{\xi _{2}\xi _{3}}\widehat{\mathcal{H}}_{p}(\xi )=p(p-1)\sum_{j=0}^{p-2}\sum_{i=0}^{j}C_{p-2}^{j}C_{j}^{i}\theta
_{i}\sigma _{j+1}\xi _{1}^{i}\xi _{2}^{j-i}\xi _{3}^{\left( p-2\right) -j}.
\end{equation*}
\end{lemma}

\begin{proof}
It is sufficient to follow the proof steps provided in \cite{Kouachi2002,Abdelmalek2007}.
\end{proof}
\subsection{Theorem \ref{main_Theor}: Global existence}
\label{Thrm1_GE}
The key element of the proof of Theorem \ref{main_Theor} is a Lyapunov functional introduced in the following proposition.
\begin{proposition}
\label{Prop1} Let $p$ be any positive integer, the positive real numbers $\theta$ and $\sigma $ satisfy  (\ref{Teta_Sigma_Conds}), (\ref{Cond_p}). Assume that $f,g,h$ satisfy \ref{a1} with \ref{a2}, and \ref{a3} for some $%
\lambda _{1}\geq $ $\theta ^{2(p-1)}$, $\lambda _{2}\geq $ $\sigma ^{2(p-1)}$%
. Let $S$ be a maximal solution of system (\ref{main_system}), and let
\begin{equation*}
\mathcal{L}(t)=\int_{\Omega }\mathcal{H}_{p}(S(x,t))dx,
\end{equation*}%
where%
\begin{equation*}
\mathcal{H}_{p}(S)=\sum_{j=0}^{p}\sum_{i=0}^{j}C_{p}^{j}C_{j}^{i}\theta
^{i^{2}-i}\sigma ^{j^{2}-j}u^{i}v^{j-i}w^{p-j}.
\end{equation*}%
Then the functional $\mathcal{L}$ is uniformly bounded on the interval $%
[0,T^{\ast }]$, where $T^{\ast }<T_{max}$.
\end{proposition}

\begin{proof}
Set $\varphi (t,x)=\theta v(t,x)$, $\psi (t,x)=\theta \sigma w(t,x)$ and $%
\widetilde{\mathcal{L}}(t)=\left( \theta \sigma \right) ^{p}\mathcal{L}(t)$.
We obtain%
\begin{equation*}
(\theta \sigma )^{p}\mathcal{H}_{p}(S)=\sum_{j=0}^{p}%
\sum_{i=0}^{j}C_{p}^{j}C_{j}^{i}\theta ^{i^{2}}\sigma ^{j^{2}}u^{i}\varphi
^{j-i}\psi ^{p-j}=:\widetilde{\mathcal{H}}_{p}(u,\varphi ,\psi ),
\end{equation*}%
where $\varphi (t,x)$ and $\psi (t,x)$ solve $\frac{\partial \varphi }{%
\partial t}-d_{2}\Delta \varphi =\theta g(S)$, $\frac{\partial \psi }{%
\partial t}-d_{3}\Delta \psi =\theta \sigma h(S)$, respectively.

Taking the derivative of $\widetilde{\mathcal{L}}$ with respect to $t$ gives%
\begin{eqnarray*}
\widetilde{\mathcal{L}}^{\prime }(t) &=&\int_{\Omega }\left( \partial _{u}%
\widetilde{\mathcal{H}}_{p}(u,\varphi ,\psi )\frac{\partial u}{\partial t}%
+\partial _{\varphi }\widetilde{\mathcal{H}}_{p}(u,\varphi ,\psi )\frac{%
\partial \varphi }{\partial t}+\partial _{\psi }\widetilde{\mathcal{H}}%
_{p}(u,\varphi ,\psi )\frac{\partial \psi }{\partial t}\right) \,dx \\
&=&\int_{\Omega }\left( d_{1}\partial _{u}\widetilde{\mathcal{H}}%
_{p}(u,\varphi ,\psi )\Delta u+d_{2}\partial _{\varphi }\widetilde{\mathcal{H%
}}_{p}(u,\varphi ,\psi )\Delta \varphi +d_{3}\partial _{\psi }\widetilde{%
\mathcal{H}}_{p}(u,\varphi ,\psi )\Delta \psi \right) dx \\
&&+\int_{\Omega }\left( f\left( u,v,w\right) \partial _{u}\widetilde{%
\mathcal{H}}_{p}(u,\varphi ,\psi )+\theta g(u,v,w)\partial _{\varphi }%
\widetilde{\mathcal{H}}_{p}(u,\varphi ,\psi )+\theta \sigma h(u,v,w)\partial
_{\psi }\widetilde{\mathcal{H}}_{p}(u,\varphi ,\psi )\right) dx \\
&=&\mathcal{I}+\mathcal{J}.
\end{eqnarray*}%
Utilizing Green's formula together with the homogeneous Neumann boundary
condition, we derive%
\begin{eqnarray*}
\mathcal{I} &\mathcal{=}&-\int_{\Omega }d_{1}\left( \partial^{2} _{u^{2}}%
\widetilde{\mathcal{H}}_{p}\right) \left\vert \nabla u\right\vert
^{2}+\left( d_{1}+d_{2}\right) \left( \partial^{2} _{u\varphi }\widetilde{%
\mathcal{H}}_{p}\right) \nabla u\nabla \varphi +\left( d_{1}+d_{3}\right)
\left( \partial^{2} _{u\psi }\widetilde{\mathcal{H}}_{p}\right) \nabla u\nabla
\psi dx \\
&&-\int_{\Omega }d_{2}\left( \partial^{2} _{\varphi ^{2}}\widetilde{\mathcal{H}}%
_{p}\right) \left\vert \nabla \varphi \right\vert ^{2}+\left(
d_{2}+d_{3}\right) \left( \partial^{2} _{\varphi \psi }\widetilde{\mathcal{H}}%
_{p}\right) \nabla \varphi \nabla \psi +d_{3}\left( \partial^{2} _{\psi ^{2}}%
\widetilde{\mathcal{H}}_{p}\right) \left\vert \nabla \psi \right\vert ^{2}dx.
\end{eqnarray*}%
By invoking lemma \ref{LemmaD2Hp}, we arrive at%
\begin{equation}
\label{I_Lyap_Func}
\mathcal{I=-}p(p-1)\int_{\Omega
}\sum_{j=0}^{p-2}\sum_{i=0}^{j}C_{p-2}^{j}C_{j}^{i}\left( \mathcal{V}^{T}%
\left( \mathcal{B}_{ij}\otimes I_N \right)\mathcal{V}\right) u^{i}\varphi ^{j-i}\psi ^{\left(
p-2\right) -j}dx,
\end{equation}%
where $I_N$ represents the identity matrix, $\mathcal{V}=\left(\nabla u , \nabla \varphi  , \nabla \psi \right)^{T}$, and 
\begin{equation*}
\mathcal{B}_{ij}=\left( 
\begin{array}{ccc}
d_{1}\theta ^{\left( i+2\right) ^{2}}\sigma ^{\left( j+2\right) ^{2}} & 
\frac{d_{1}+d_{2}}{2}\theta ^{\left( i+1\right) ^{2}}\sigma ^{\left(
j+2\right) ^{2}} & \frac{d_{1}+d_{3}}{2}\theta ^{\left( i+1\right)
^{2}}\sigma ^{\left( j+1\right) ^{2}} \\ 
\frac{d_{1}+d_{2}}{2}\theta ^{\left( i+1\right) ^{2}}\sigma ^{\left(
j+2\right) ^{2}} & d_{2}\theta ^{i^{2}}\sigma ^{\left( j+2\right) ^{2}} & 
\frac{d_{2}+d_{3}}{2}\theta ^{i^{2}}\sigma ^{\left( j+1\right) ^{2}} \\ 
\frac{d_{1}+d_{3}}{2}\theta ^{\left( i+1\right) ^{2}}\sigma ^{\left(
j+1\right) ^{2}} & \frac{d_{2}+d_{3}}{2}\theta ^{i^{2}}\sigma ^{\left(
j+1\right) ^{2}} & d_{3}\theta ^{i^{2}}\sigma ^{j^{2}}
\end{array}%
\right).
\end{equation*}%
According to the assumption (\ref{Teta_Sigma_Conds}), all the principal minors in the
top-left corner of $\mathcal{B}_{ij}$ are positive. Consequently, all the
quadratic forms $\mathcal{V}^{T}%
\left( \mathcal{B}_{ij}\otimes I_N \right)\mathcal{V} $,
for $i=0,...,j$ and  $j=0,...,p-2$, are positive definite. Thus, for all $%
t\in \lbrack 0,T_{max})$, we get 
\begin{equation*}
\mathcal{I\leq }0.
\end{equation*}%
On the other hand,
\begin{eqnarray}
\label{J_Lyap_Func}
\mathcal{J} &=&\int_{\Omega }p\sum_{j=0}^{p-1}\sum_{i=0}^{j}\left[
C_{p-1}^{j}C_{j}^{i}u^{i}\varphi ^{j-i}\psi ^{(p-1)-j}\right] \notag\\
&&\times\left( \theta
^{(i+1)^{2}}\sigma ^{(j+1)^{2}}f\left( S\right) + \theta^{i^{2}+1}\sigma ^{(j+1)^{2}}g\left( S\right) + \theta^{i^{2}+1} \sigma^{j^{2}+1}h\left( S\right) \right) \,dx \notag\\
&=&\int_{\Omega }p\sum_{j=0}^{p-1}\sum_{i=0}^{j}\left[
C_{p-1}^{j}C_{j}^{i}u^{i}\varphi ^{j-i}\psi ^{(p-1)-j}\right]  \theta^{i^{2}+1} \sigma ^{j^{2}+1} \notag\\
&&\times\left( \frac{\theta ^{(i+1)^{2}}}{\theta
^{i^{2}+1}}\frac{\sigma ^{(j+1)^{2}}}{\sigma ^{j^{2}+1}}f\left( S\right)
+\frac{\sigma ^{(j+1)^{2}}}{\sigma ^{j^{2}+1}}g\left( S\right) +h\left(
S\right) \right) \,dx \notag\\
&=&\int_{\Omega }p\sum_{j=0}^{p-1}\sum_{i=0}^{j}\left[
C_{p-1}^{j}C_{j}^{i}u^{i}\varphi ^{j-i}\psi ^{(p-1)-j}\right] \theta^{i^{2}+1} \sigma^{j^{2}+1} \\
&&\times\left( \theta ^{2i}\sigma ^{2j}f\left(
S\right) +\sigma ^{2j}g\left( S\right) +h\left( S\right) \right)
\,dx. \notag
\end{eqnarray}%
In the light of the conditions \ref{a2}, \ref{a3} for some $\lambda _{1}\geq 
$ $\theta ^{2(p-1)}$, $\lambda _{2}\geq $ $\sigma ^{2(p-1)}$ and Lemma \ref%
{LIWS_lemma}, we get%
\begin{eqnarray*}
\mathcal{J} &\leq &\widetilde{K}\int_{\Omega
}p\sum_{j=0}^{p-1}\sum_{i=0}^{j}C_{p-1}^{j}C_{j}^{i}u^{i}\varphi ^{j-i}\psi
^{(p-1)-j}\Lambda(S)dx \\
&=&\widetilde{K}\int_{\Omega }p\sum_{j=0}^{p-1}\sum_{i=0}^{j}\theta ^{j-i}\sigma
^{p-j-1}C_{p-1}^{j}C_{j}^{i}\Lambda(S)u^{i}v^{j-i}w^{p-j-1}dx,
\end{eqnarray*}%
where $\widetilde{K}\geq 0$. Using a similar approach as in \cite{Kouachi2001}, it follows that there exist
positive constants $M_{1}$ and $M_{2}$ such that the functional $\widetilde{%
\mathcal{L}}$ satisfies the differential inequality 
\begin{equation}
\widetilde{\mathcal{L}}^{\prime }(t)\leq M_{1}\widetilde{\mathcal{L}}%
(t)+M_{2}\widetilde{\mathcal{L}}^{\frac{p-1}{p}}(t) . \label{L_inqD}
\end{equation}%
By setting $\mathcal{E=}\widetilde{\mathcal{L}}^{\frac{1}{p}}$, then \ref%
{L_inqD} can be rewritten as%
\begin{equation}
p\mathcal{E}^{\prime }(t)\leq M_{1}\mathcal{E}(t)+M_{2}.  \label{E_inqD}
\end{equation}%
Thus, a straightforward integration of (\ref{E_inqD}) yields a uniform bound
for the functional $\widetilde{\mathcal{L}}$ over the interval $[0,T^{\ast
}] $, with $T^{\ast }<T_{max}$, i.e., there exists a positive constant $%
\alpha $ such that 
\begin{equation}
\left( \theta \sigma \right) ^{p}\mathcal{L}(t)=\widetilde{\mathcal{L}}%
\left( t\right) \leq \alpha ,\ \forall t\in \lbrack 0,T^{\ast }],\ T^{\ast
}\in \left( 0,T_{max}\right) .
\end{equation}%
Hence, the functional $\mathcal{L}$ is upper bounded by $\beta :=\alpha
\left( \theta \sigma \right) ^{-p}$ on the interval $[0,T^{\ast}]$, with $T^{\ast }<T_{max}$.
\end{proof}
Now, since 
\begin{equation}
\Vert u(.,t)+v(.,t)+w(.,t)\Vert _{L^{p}(\Omega )}^{p}\leq \mathcal{L}%
(t),\quad \text{ on }\left[ 0,T_{max}\right),
\end{equation}%
by virtue of Proposition \ref{Prop1}, for any positive integer $p$, we
obtain 
\begin{equation}
 u+v+w \in L^{\infty }\left( \left( 0,T_{max}\right)
;L^{p}(\Omega )\right) .
\end{equation}%
Thus,
\begin{equation}
\Lambda \left( S\right) \in L^{\infty }\left( \left( 0,T_{max}\right)
;L^{p}(\Omega )\right) .  \label{uvp_bound}
\end{equation}%
On the other hand, from \ref{a4}, we get%
\begin{equation}
\left\{ 
\begin{array}{c}
\Vert f\left( S(.,t)\right) \Vert _{L^{q}(\Omega )}^{q}\leq
\gamma \Vert \Lambda\left( S(.,t)\right) \Vert _{L^{p}(\Omega )}^{p}, \\ \\
\Vert g\left( S(.,t)\right) \Vert _{L^{q}(\Omega )}^{q}\leq
\gamma \Vert \Lambda\left( S(.,t)\right)\Vert _{L^{p}(\Omega )}^{p}, \\ \\
\Vert h\left( S(.,t)\right) \Vert _{L^{q}(\Omega )}^{q}\leq
\gamma \Vert \Lambda\left( S(.,t)\right)\Vert _{L^{p}(\Omega )}^{p},%
\end{array}%
\right. \text{ on }\left( 0,T_{max}\right),
\label{FGH_q_bound}
\end{equation}%
where $q=\frac{p}{m}$ and $\gamma >0$. Finally, by combaining (\ref%
{FGH_q_bound}), (\ref{uvp_bound}), (\ref{Cond_p}) and using the $L^{p}$%
-regularity of the heat operator (see \cite{Kirane1983,Ladyzenskaja1968}),
we conclude that the system (\ref{main_system}) has a unique nonnegative
global classical solution.
\subsection{Theorem \ref{main_Theor}: Uniform-in-time bounds}
\label{Thrm1_UBt}
We first show that for any $1\leq p<\infty$, there exists a constant $C_p>0$ such that
	\begin{equation}\label{uvw_LpUniformUpperBound}
		\sup_{t\geq 0}\|u(.,t)\|_{L^{p}(\Omega)},\sup_{t\geq 0}\|v(.,t)\|_{L^{p}(\Omega)},\sup_{t\geq 0}\|w(.,t)\|_{L^{p}(\Omega)} \leq C_p.
	\end{equation}
    Indeed, by employing the $L^p$-energy function $\mathcal{L}(t)$ together with $\widetilde{\mathcal{L}}(t)=(\theta \sigma)^p \mathcal{L}(t)$, and carrying out computations analogous to those in Proposition~\ref{Prop1}, while taking into account the assumptions \ref{a2} and \ref{a3} with $K_i=0$, for all $i=1,\dots,4$, we obtain
\begin{equation}
    \widetilde{\mathcal{L}}^{\prime}(t)\leq 0.
\end{equation}
Thus, there exists a positive constant $\widehat{C}$, such that
\begin{equation*}
		\sup_{t\geq 0}\mathcal{L}(t) \leq \widehat{C}.
	\end{equation*}
Hence, we get \eqref{uvw_LpUniformUpperBound}. To finally see that the solution is bounded uniformly in time in sup norm, we use a smooth cut-off function $\psi: \mathbb{R} \to [0,1]$ with $\psi(s) = 0$ for $s\leq 0$ and $\psi(s) = 1$ for $s\geq 1$, $0\leq \psi' \leq M$, with $M>0$, and its shifted version $\psi_\tau(\cdot) = \psi(\cdot-\tau)$ for any $\tau\in \mathbb N$. Let $\tau\in\mathbb N$ be arbitrary. By multiplying the system \eqref{main_system} by $\psi_\tau$, we obtain
\begin{equation}
\left\{ 
\begin{array}{cc}
\displaystyle\frac{\partial (\psi_\tau u)}{\partial t}-d_{1}\Delta (\psi_\tau u)=\psi_\tau^{\prime}u+\psi_\tau f\left( u,v,w\right) , & \text{%
in }\Omega \times \left( \tau,\tau+\rho \right) , \\ \\
\displaystyle\frac{\partial (\psi_\tau v)}{\partial t}-d_{2}\Delta (\psi_\tau v)=\psi_\tau^{\prime}v+\psi_\tau g\left( u,v,w\right) , & \text{%
in }\Omega \times \left( \tau,\tau+\rho \right) , \\ \\
\displaystyle\frac{\partial (\psi_\tau w)}{\partial t}-d_{3}\Delta (\psi_\tau w)=\psi_\tau^{\prime}w+\psi_\tau h\left( u,v,w\right) , & \text{%
in }\Omega \times \left( \tau,\tau+\rho \right) , \\ \\
\displaystyle \frac{\partial (\psi_\tau u)}{\partial \eta }\left( x,t\right) =\frac{\partial (\psi_\tau v)}{%
\partial \eta }\left( x,t\right) =\frac{\partial (\psi_\tau w)}{\partial \eta }\left(
x,t\right) =0, & \text{on }\partial \Omega \times \left( \tau,\tau+\rho \right) ,
\\ \\
(\psi_\tau u)\left( x,\tau\right) = (\psi_\tau v)\left( x,\tau\right) =(\psi_\tau w)\left( x,\tau\right) =0, & \text{in }\Omega ,%
\end{array}%
\right.  \label{main_system_Cutoff}
\end{equation}
where $2\leq\rho\in \mathbb{N}$. Thanks to \eqref{uvw_LpUniformUpperBound} and the polynomial growth \ref{a4}, we have 
	\begin{equation*}
		\psi_\tau^{\prime}u+\psi_\tau f\left(S\right),\quad  \psi_\tau^{\prime}v+\psi_\tau g\left(S\right),\quad  \psi_\tau^{\prime}w+\psi_\tau h\left(S\right)\leq \widetilde{\Lambda}(S):=C (\Lambda(S))^m.
	\end{equation*}
	Again thanks to \eqref{uvw_LpUniformUpperBound} for any $1\leq p<\infty$, a constant $C_p>0$ exists such that
	\begin{equation*}
		\|\widetilde{\Lambda}(S)\|_{L^{p}(\Omega\times(\tau,\tau+\rho))} \leq C_p.
	\end{equation*}
	Therefore, by choosing $p>\frac{N+2}{2}$ and using the smooth effect of parabolic operator (see e.g. \cite[Proposition 3.1]{Nittka2014}),  we get
	\begin{equation*}
		\|\psi_\tau u\|_{L^{\infty}(\Omega\times(\tau,\tau+\rho))},\|\psi_\tau v\|_{L^{\infty}(\Omega\times(\tau,\tau+\rho))},\|\psi_\tau w\|_{L^{\infty}(\Omega\times(\tau,\tau+\rho))} \leq C,
	\end{equation*}
	where $C$ is a constant {\it independent of $\tau\in \mathbb N$}. Thanks to $\psi_\tau \geq 0$ and $\psi|_{(\tau+1,\tau+\rho)} \equiv 1$, we obtain finally the uniform-in-time bound
	\begin{equation*}
		\sup_{t\geq 0}\|u(.,t)\|_{L^{\infty}(\Omega)},\quad \sup_{t\geq 0}\|v(.,t)\|_{L^{\infty}(\Omega)},\quad \sup_{t\geq 0}\|w(.,t)\|_{L^{\infty}(\Omega)} \leq C.
	\end{equation*}

\subsection{Proof of Theorem \ref{main_Theor2}}
\label{Thrm2_Proof}
The key element of the proof of Theorem \ref{main_Theor2} is the following
Lyapunov functional.

\begin{proposition}
\label{Prop2} Let $p$, $\theta $, and $\sigma $
satisfy the same assumptions as in Proposition \ref{Prop1}. \ Assume that $%
f,g,h$ fulfill conditions \ref{a1}, \ref{a2}, and \ref{a3} for some $\lambda
_{1}\leq $ $\theta ^{-2p}$, $\lambda _{2}\leq $ $\sigma ^{-2p}$. Let $S $ be a maximal solution of system (\ref{main_system}), and let
\begin{equation*}
\mathcal{L}_{1}(t)=\int_{\Omega }\mathcal{H}_{p}^{\ast
}(S(x,t))dx,
\end{equation*}%
where%
\begin{equation*}
\mathcal{H}_{p}^{\ast
}(S)=\sum_{j=0}^{p}\sum_{i=0}^{j}C_{p}^{j}C_{j}^{i}\theta ^{\left(
p-i\right) ^{2}-i}\sigma ^{\left( p-j\right) ^{2}-j}u^{i}v^{j-i}w^{p-j}.
\end{equation*}%
Then, the functional $\mathcal{L}_{1}$ is uniformly bounded on the interval $%
[0,T^{\ast }]$, where $T^{\ast }<T_{max}$. \ 
\end{proposition}

\begin{proof}[Proof of Proposition \ref{Prop2}]
Set $\varphi (t,x)=\theta v(t,x)$, $\psi (t,x)=\theta \sigma w(t,x)$ and $%
\widetilde{\mathcal{L}}_{1}(t)=\left( \theta \sigma \right) ^{p}\mathcal{L}%
_{1}(t)$. We obtain%
\begin{equation*}
(\theta \sigma )^{p}\mathcal{H}_{p}^{\ast
}(S)=\sum_{j=0}^{p}\sum_{i=0}^{j}C_{p}^{j}C_{j}^{i}\theta ^{\left(
p-i\right) ^{2}}\sigma ^{\left( p-j\right) ^{2}}u^{i}\varphi ^{j-i}\psi
^{p-j}=:\widetilde{\mathcal{H}}_{p}^{\ast }(u,\varphi ,\psi ),
\end{equation*}%
where $\varphi (t,x)$ and $\psi (t,x)$ solve $\frac{\partial \varphi }{%
\partial t}-d_{2}\Delta \varphi =\theta g(u,v,w)$, $\frac{\partial \psi }{%
\partial t}-d_{3}\Delta \psi =\theta \sigma h(u,v,w)$, respectively.
Taking the derivative of $\widetilde{\mathcal{L}}$ with respect to $t$ gives%
\begin{eqnarray}
\widetilde{\mathcal{L}}_{1}^{\prime }(t) &=&\int_{\Omega }\left( \partial _{u}%
\widetilde{\mathcal{H}}_{p}^{\ast }(u,\varphi ,\psi )\frac{\partial u}{%
\partial t}+\partial _{\varphi }\widetilde{\mathcal{H}}_{p}^{\ast
}(u,\varphi ,\psi )\frac{\partial \varphi }{\partial t}+\partial _{\psi }%
\widetilde{\mathcal{H}}_{p}^{\ast }(u,\varphi ,\psi )\frac{\partial \psi }{%
\partial t}\right) \,dx \notag\\
&=&\int_{\Omega }\left( d_{1}\partial _{u}\widetilde{\mathcal{H}}_{p}^{\ast
}(u,\varphi ,\psi )\Delta u+d_{2}\partial _{\varphi }\widetilde{\mathcal{H}}%
_{p}^{\ast }(u,\varphi ,\psi )\Delta \varphi +d_{3}\partial _{\psi }%
\widetilde{\mathcal{H}}_{p}^{\ast }(u,\varphi ,\psi )\Delta \psi \right) dx  \notag\\
&&+\int_{\Omega }\left( f\left( S\right) \partial _{u}\widetilde{%
\mathcal{H}}_{p}^{\ast }(u,\varphi ,\psi )+\theta g(S)\partial _{\varphi
}\widetilde{\mathcal{H}}_{p}^{\ast }(u,\varphi ,\psi )+\theta \sigma
h(S)\partial _{\psi }\widetilde{\mathcal{H}}_{p}^{\ast }(u,\varphi ,\psi
)\right) dx \notag\\
&=&\mathcal{I}^{\ast }+\mathcal{J}^{\ast }.
\label{IJ2_Lyap_Func}
\end{eqnarray}%
On one hand, similar to the proof of Proposition \ref{Prop1}, we derive for
all $t\in \lbrack 0,T_{max})$%
\begin{equation*}
\mathcal{I}^{\ast }\mathcal{\leq }0.
\end{equation*}%
On the other hand, we have 
\begin{eqnarray*}
\mathcal{J}^{\ast } &=&\int_{\Omega }p\sum_{j=0}^{p-1}\sum_{i=0}^{j}\left[
C_{p-1}^{j}C_{j}^{i}u^{i}\varphi ^{j-i}\psi ^{(p-1)-j}\right]  \\
&&\times \left( \theta ^{(p-i-1)^{2}}\sigma ^{(p-j-1)^{2}}f\left(S\right) +
 \theta ^{(p-i)^{2}+1}\sigma ^{(p-j-1)^{2}}g\left(
S\right) + \theta^{(p-i)^{2}+1} \sigma^{(p-j)^{2}+1}h\left(
S\right) \right) \,dx \\
&=&\int_{\Omega }p\sum_{j=0}^{p-1}\sum_{i=0}^{j}\left[
C_{p-1}^{j}C_{j}^{i}u^{i}\varphi ^{j-i}\psi ^{(p-1)-j}\right] 
 \theta^{(p-i)^{2}+1} \sigma^{(p-j)^{2}+1} \\
&&\times \left( \frac{\theta ^{(p-i-1)^{2}}}{\theta ^{(p-i)^{2}+1}}\frac{%
\sigma ^{(p-j-1)^{2}}}{\sigma ^{(p-j)^{2}+1}}f\left(S\right) +\frac{%
\sigma ^{(p-j-1)^{2}}}{\sigma ^{(p-j)^{2}+1}}g\left( S\right) +h\left(
S\right) \right) \,dx \\
&=&\int_{\Omega }p\sum_{j=0}^{p-1}\sum_{i=0}^{j}\left[
C_{p-1}^{j}C_{j}^{i}u^{i}\varphi ^{j-i}\psi ^{(p-1)-j}\right] 
 \theta^{(p-i)^{2}+1} \sigma^{(p-j)^{2}+1} \\
&&\times \left( \theta ^{2(i-p)}\sigma ^{2(j-p)}f\left( S\right) +\sigma
^{2(j-p)}g\left( S\right) +h\left(S\right) \right) \,dx.
\end{eqnarray*}%
In the light of the conditions \ref{a2}, \ref{a3} for some $\lambda _{1}\leq 
$ $\theta ^{-2p}$, $\lambda _{2}\leq $ $\sigma ^{-2p}$ and using Lemma \ref%
{LIWS_lemma}, we get 
\begin{eqnarray*}
\mathcal{J}^{\ast } &\leq &\hat{K}\int_{\Omega
}p\sum_{j=0}^{p-1}\sum_{i=0}^{j}C_{p-1}^{j}C_{j}^{i}u^{i}\varphi ^{j-i}\psi
^{(p-1)-j}\Lambda(S)dx \\
&=&\hat{K}\int_{\Omega }p\sum_{j=0}^{p-1}\sum_{i=0}^{j}\theta ^{j-i}\sigma
^{p-j-1}C_{p-1}^{j}C_{j}^{i}\Lambda(S)u^{i}v^{j-i}w^{p-j-1}dx,
\end{eqnarray*}%
where $\hat{K}\geq 0$. Now, by following the remaining steps in the proof of
Proposition \ref{Prop1}, there exists a positive constant $\alpha ^{\ast }$,
such that 
\begin{equation}
\left( \theta \sigma \right) ^{p}\mathcal{L}_{1}(t)=\widetilde{\mathcal{L}}%
_{1}\left( t\right) \leq \alpha ^{\ast } ,\ \forall t\in \lbrack 0,T^{\ast }],\
T^{\ast }\in \left( 0,T_{max}\right) .
\end{equation}%
Hence, the functional $\mathcal{L}_{1}$ is upper bounded by $\beta ^{\ast
}:=\alpha ^{\ast }\left( \theta \sigma \right) ^{-p}$ on the interval $%
[0,T^{\ast }]$, with $T^{\ast }<T_{max}$.
\end{proof}
Now, since 
\begin{equation}
\Vert u(.,t)+v(.,t)+w(.,t)\Vert _{L^{p}(\Omega )}^{p}\leq \left( \theta
\sigma \right) ^{p}\mathcal{L}_{1}(t),\quad \text{ on }\left[
0,T_{max}\right),
\end{equation}%
by virtue of Proposition \ref{Prop2}, for any positive integer $p$, we
obtain 
\begin{equation}
 u+v+w\in L^{\infty }\left( \left( 0,T_{max}\right)
;L^{p}(\Omega )\right) .
\end{equation}%
Then, the rest of the proof is the same as in the proof of Theorem \ref{main_Theor}. Hence, we conclude that the system (\ref{main_system}) has a
unique nonnegative global classical solution.

\subsection{Generalizations: Proof of Theorems \ref{main_Theor3GBCs} and \ref{main_Theor4GBCs}.}
\label{Sub_sec_GeneralizationsProof}
\begin{proof}[Proof of Theorem \ref{main_Theor3GBCs}]
    The proof proceeds in a manner similar to that of Theorem \ref{main_Theor}, differing only in that
    \begin{equation}
        \widetilde{\mathcal{L}}^{\prime }(t)=\mathcal{I}+\widehat{\mathcal{I}}+\mathcal{J},
    \end{equation}
where $\mathcal{I}$ and $\mathcal{J}$ are the same as in \eqref{I_Lyap_Func} and \eqref{J_Lyap_Func}, respectively. The additional term is given by 
\begin{equation}
    \widehat{\mathcal{I}}=\displaystyle\int_{\partial \Omega }\left( d_{1}\partial _{u}\widetilde{\mathcal{H}}%
_{p}(u,\varphi ,\psi ) \frac{\partial u}{\partial \eta }+d_{2}\partial _{\varphi }\widetilde{\mathcal{H%
}}_{p}(u,\varphi ,\psi )\frac{\partial \varphi}{\partial \eta }  +d_{3}\partial _{\psi }\widetilde{\mathcal{H}}_{p}(u,\varphi ,\psi )\frac{\partial \psi}{\partial \eta } \right) dx.
\end{equation}
We now seek to show that, there exists $C\geq 0$, such that
\begin{equation}
\label{IBCs_Lyap_Func}
    \widehat{\mathcal{I}}\leq C.
\end{equation}
Indeed, we follow the steps as in \cite[Page 7]{Kouachi2002}. Thus, we get \eqref{IBCs_Lyap_Func} with $C>0$ for non-homogeneous or mixed homogeneous and non-homogeneous boundary conditions, as well as we have $C=0$ for homogeneous boundary conditions.
\end{proof}

\begin{proof}[Proof of Theorem \ref{main_Theor4GBCs}]
    The proof proceeds in a manner similar to that of Theorem \ref{main_Theor2}, differing only in that
    \begin{equation}
        \widetilde{\mathcal{L}}_{1}^{\prime }(t)=\mathcal{I}^{\ast}+\widehat{\mathcal{I}}^{\ast}+\mathcal{J}^{\ast},
    \end{equation}
where $\mathcal{I}^{\ast}$ and $\mathcal{J}^{\ast}$ are the same as in \eqref{IJ2_Lyap_Func}. The additional term is given by 
\begin{equation}
\widehat{\mathcal{I}}^{\ast}=\displaystyle\int_{\partial \Omega }\left( d_{1}\partial _{u}\widetilde{\mathcal{H}}_{p}^{\ast
}(u,\varphi ,\psi ) \frac{\partial u}{\partial \eta }+d_{2}\partial _{\varphi }\widetilde{\mathcal{H}}%
_{p}^{\ast }(u,\varphi ,\psi )\frac{\partial \varphi}{\partial \eta }  +d_{3}\partial _{\psi }%
\widetilde{\mathcal{H}}_{p}^{\ast }(u,\varphi ,\psi )\frac{\partial \psi}{\partial \eta } \right) dx.
\end{equation}

We now seek to show that, there exists $\widetilde{C}\geq 0$, such that
\begin{equation}
\label{IBCs_Lyap_Func1}
    \widehat{\mathcal{I}}^{\ast}\leq \widetilde{C}.
\end{equation}
Indeed, we follow the steps as in \cite[Page 7]{Kouachi2002}. Thus, we get \eqref{IBCs_Lyap_Func1} with $\widetilde{C}>0$ for non-homogeneous or mixed homogeneous and non-homogeneous boundary conditions, as well as we have $\widetilde{C}=0$ for homogeneous boundary conditions.
\end{proof}
\section{Applications}
\label{Sec_Applications}

\subsection{Example 1}
\label{Sec_Applications_Ex1}
Consider the example:%
\begin{equation}
\left\{ 
\begin{array}{cc}
\displaystyle\frac{\partial u}{\partial t}-d_{1}\Delta u=v^{l}-u^{q}, & \text{in }\Omega
\times \left( 0,+\infty \right) , \\ \\
\displaystyle\frac{\partial v}{\partial t}-d_{2}\Delta v=w^{r}-Bv^{l}, & \text{in }\Omega
\times \left( 0,+\infty \right) , \\ \\
\displaystyle\frac{\partial w}{\partial t}-d_{3}\Delta w=Au^{q}-Cw^{r}, & \text{in }%
\Omega \times \left( 0,+\infty \right) , \\ \\
\displaystyle\frac{\partial u}{\partial \eta }\left( x,t\right) =\frac{\partial v}{%
\partial \eta }\left( x,t\right) =\frac{\partial w}{\partial \eta }\left(
x,t\right) =0, & \text{on }\partial \Omega \times \left( 0,+\infty \right) ,
\\ \\
u\left( x,0\right) =u_{0}(x)\text{, }v\left( x,0\right) =v_{0}(x)\text{, }%
w\left( x,0\right) =w_{0}(x), & \text{in }\Omega ,%
\end{array}%
\right.  \label{Sys_Ex1}
\end{equation}
where $l,q,r\geq 1$, $A$,$B$, and $C$ are positive constants such that $%
A<\min \left\{ B,C\right\}$. The present example is inspired by the system of two equations considered by Kouachi in \cite[(4.11)-(4.12)]{Kouachi2001}.
We have the following result.
\begin{theorem}
Let $A\leq 1$, $B\geq \theta ^{2(p-1)}$, $C\geq \sigma ^{2(p-1)}$, such that 
$p>\frac{\left( N+2\right) }{2}\max \{l.q.r\}$, and $\theta ,\sigma $
satisfy $\left( \text{\ref{Teta_Sigma_Conds}}\right) $. Then for any $%
u_{0},v_{0},w_{0}\in L^{\infty }\left( \Omega ;\mathbb{R}_{+}\right) $,
there exists a unique nonnegative global classical solution to $\left( \text{%
\ref{Sys_Ex1}}\right) $. Moreover, the solution is bounded uniformly in time, i.e. 
\begin{equation}\label{uvw_LinftyUniformUpperBound_GBCsEX1}
		\sup_{t\geq 0}\|u(.,t)\|_{L^{\infty}(\Omega)},\quad\sup_{t\geq 0}\|v(.,t)\|_{L^{\infty}(\Omega)},\quad \sup_{t\geq 0}\|w(.,t)\|_{L^{\infty}(\Omega)} < +\infty.
\end{equation}
\end{theorem}

\begin{proof}
As all the assumptions of Theorem \ref{main_Theor} are fulfilled. Then, the
global existence and uniform boundedness in time of the solution to system $\left( \text{\ref{Sys_Ex1}}%
\right) $ follows directly.
\end{proof}

\subsection{Example 2}
\label{Sec_Applications_Ex2}
Consider the following generalized example: 
\begin{equation}
\left\{ 
\begin{array}{ll}
\displaystyle\frac{\partial u}{\partial t}-d_{1}\Delta u=f\left( S\right) & \text{in }%
\Omega \times \left( 0,\infty \right) ,\smallskip \\ \\
\displaystyle \frac{\partial v}{\partial t}-d_{2}\Delta v=g\left( S\right) & \text{in }%
\Omega \times \left( 0,\infty \right) ,\smallskip \\ \\
\displaystyle \frac{\partial w}{\partial t}-d_{3}\Delta w=h\left(S\right) & \text{in }%
\Omega \times \left( 0,\infty \right) ,\smallskip \\ \\
\displaystyle\frac{\partial u}{\partial \eta }\left( x,t\right) =\frac{\partial v}{%
\partial \eta }\left( x,t\right) =\frac{\partial w}{\partial \eta }\left(
x,t\right) =0 & \text{on }\partial \Omega \times \left( 0,\infty \right)
,\smallskip \\ \\
\begin{array}{l}
u\left( x,0\right) =u_{0}(x)\geqslant 0\text{, }v\left( x,0\right)
=v_{0}(x)\geqslant 0, \\ 
w\left( x,0\right) =w_{0}(x)\geqslant 0%
\end{array}
& \text{in }\Omega ,%
\end{array}%
\right.  \label{Sys_Ex2}
\end{equation}%
where,%
\begin{equation}
\left\{ 
\begin{array}{c}
f(S):=\Psi _{2}(S)-\Psi _{1}(S)-s_{1}\Psi _{3}(S), \\ \\
g(S):=\Psi _{3}(S)-\zeta _{2}\Psi _{2}(S)-s_{2}\Psi _{1}(S),
\\ \\
h(S):=\zeta _{1}\Psi _{1}(S)-\zeta _{3}\Psi _{3}(S)-s_{3}\Psi_{2}(S),
\end{array}%
\right.  \label{ReacFuncs_Ex2}
\end{equation}
subject to the following constraints 
\begin{equation}
s_{1},s_{2},s_{3}\geq 0,\text{ }1\geq \zeta _{1} . \label{Cond1_Ex2}
\end{equation}%
\begin{equation}
\Psi _{1},\Psi _{2},\Psi _{3}\in C^{1}\left( \mathbb{R}_{+}^{3};\mathbb{R}%
_{+}\right)  .\label{Cond2_Ex2}
\end{equation}%

\begin{equation}
\left\{ 
\begin{array}{c}
\Psi _{1}(0,v,w)=\Psi _{2}(0,v,w)=\Psi _{3}(0,v,w)=0,\text{ for all }
(v,w)\in \mathbb{R}_{+}^{2}, \\ \\
\Psi _{1}(u,0,w)=\Psi _{2}(u,0,w)=\Psi _{3}(u,0,w)=0,\text{ for all }
(u,w)\in \mathbb{R}_{+}^{2}, \\ \\
\Psi _{1}(u,v,0)=\Psi _{2}(u,v,0)=\Psi _{3}(u,v,0)=0,\text{ for all }
(u,v)\in \mathbb{R}_{+}^{2}.%
\end{array}%
\right.  \label{Cond3_Ex2}
\end{equation}
There exist $C>0$ and $m\in \mathbb{N}$, such that
\begin{equation}
\Psi _{1}(S),\Psi
_{2}(S),\Psi _{3}(S)\leq C\Lambda(S)^{m},\ \text{ for all } S\in \mathbb{%
R}_{+}^{3}.  \label{Cond4_Ex2}
\end{equation}%
\begin{eqnarray}
p &=&\min \left\{ l\in \mathbb{N}\ :\ l>\frac{m\left( N+2\right) }{2}%
\right\} ,  \label{Cond5_Ex2} \\
\text{ }\zeta _{2} &>&\lambda _{1}\geqslant \theta ^{2\left( p-1\right) },%
\text{ }\zeta _{3}>\lambda _{2}\geqslant \sigma ^{2\left( p-1\right) }. 
\notag
\end{eqnarray}

We have the following result.

\begin{theorem}
According to the assumptions $\left( \text{\ref{Cond1_Ex2}}\right) $-$\left( 
\text{\ref{Cond5_Ex2}}\right) $ where $\left( \text{\ref{Teta_Sigma_Conds}}\right) $ holds, and for any $u_{0},v_{0},w_{0}\in L^{\infty
}\left( \Omega ;\mathbb{R}_{+}\right) $, there exists a unique nonnegative
global classical solution to the system (\ref{Sys_Ex2}). Moreover, the solution is bounded uniformly in time, i.e. 
\begin{equation}\label{uvw_LinftyUniformUpperBound_GBCsEX2}
		\sup_{t\geq 0}\|u(.,t)\|_{L^{\infty}(\Omega)},\quad\sup_{t\geq 0}\|v(.,t)\|_{L^{\infty}(\Omega)},\quad \sup_{t\geq 0}\|w(.,t)\|_{L^{\infty}(\Omega)} < +\infty.
\end{equation}
\end{theorem}

\begin{proof}
The reaction functions $\left( \text{\ref{ReacFuncs_Ex2}}\right) $, under
the conditions $\left( \text{\ref{Cond1_Ex2}}\right) $-$\left( \text{\ref%
{Cond5_Ex2}}\right) $, can be readily verified to satisfy the assumptions of
Theorem \ref{main_Theor}. Consequently, the global existence and uniform boundedness in time of the solution to system $\left( \text{\ref{Sys_Ex2}}\right) $ are established.
\end{proof}

\subsection{Example 3: HOIs three-species Lotka-Voltera model}
\label{Sec_Applications_Ex3}

In contrast to the previous examples, where the parameters of the systems
are dependent on the spatial dimension, the growth polynomial's order, and
the diffusion coefficients, the current example admits parameters that remain independent of these quantities. A direct
application of Theorems \ref{main_Theor3GBCs} and \ref{main_Theor4GBCs} demonstrates
the existence of global classical solution in arbitrary spatial dimensions
for nonlinear three-species sub-skew-symmetric Lotka-Volterra involving
higher order interactions and linear diffusion, expressed in the following
form: 
\begin{equation}
\left\{ 
\begin{array}{cc}
\vspace{0.2cm}
\displaystyle \frac{\partial u}{\partial t}-d_{1}\Delta u=\tau _{1}u+u^{\gamma _{1}}\left(
a_{11}u^{\gamma _{1}}+a_{12}v^{\gamma _{2}}+a_{13}w^{\gamma _{3}}\right) , & 
\text{in }\Omega \times \left( 0,+\infty \right) , \\ \vspace{0.2cm}
\displaystyle\frac{\partial v}{\partial t}-d_{2}\Delta v=\tau _{2}v+v^{\gamma _{2}}\left(
a_{21}u^{\gamma _{1}}+a_{22}v^{\gamma _{2}}+a_{23}w^{\gamma _{3}}\right) , & 
\text{in }\Omega \times \left( 0,+\infty \right) , \\ 
\displaystyle\frac{\partial w}{\partial t}-d_{3}\Delta w=\tau _{3}w+w^{\gamma _{3}}\left(
a_{31}u^{\gamma _{1}}+a_{32}v^{\gamma _{2}}+a_{33}w^{\gamma _{3}}\right) , & 
\text{in }\Omega \times \left( 0,+\infty \right) , \\ 
\begin{array}{c}
\displaystyle\lambda _{1}u(x,t)+\left( 1-\lambda _{1}\right) \frac{\partial u}{\partial
\eta }\left( x,t\right) =\beta _{1}, \\ 
\displaystyle\lambda _{2}v(x,t)+\left( 1-\lambda _{2}\right) \frac{\partial v}{\partial
\eta }\left( x,t\right) =\beta _{2}, \\ 
\displaystyle\lambda _{3}w(x,t)+\left( 1-\lambda _{3}\right) \frac{\partial w}{\partial
\eta }\left( x,t\right) =\beta _{3},%
\end{array}
& \text{on }\partial \Omega \times \left( 0,+\infty \right) , \\ 
u\left( x,0\right) =u_{0}(x)\text{, }v\left( x,0\right) =v_{0}(x)\text{, }%
w\left( x,0\right) =w_{0}(x), & \text{in }\Omega ,%
\end{array}%
\right.   \label{3SHOIs_LV}
\end{equation}
where $A=(a_{ij})\in 
\mathbb{R}
^{3\times 3}$ is sub-skew-symmetric matrix, i.e. $A+A^{T}\leq 0$ where all
elements of $A+A^{T}$ are non-positive, $(\tau _{1},\tau _{2},\tau _{3})\in 
\mathbb{R}
^{3}$, and $\gamma _{i}\geq 1$ for $i=1,2,3$. The boundary conditions are
defined as \eqref{HNBCs_EX3}-\eqref{MixNHRBCsHDBCs_EX3}.

The diffusive Lotka-Volterra model with higher order interactions $\left( 
\text{\ref{3SHOIs_LV}}\right) $, is inspired by the works \cite{AlAdwani2019,AlAdwani2020,Desvillettes2014,Desvillettes2015,Suzuki2013,Vandermeer1969}. The model $\left( \text{\ref{3SHOIs_LV}}\right) $ with $\gamma
_{i}=1$ for $i=1.2.3$, and homogeneous Neumann or Dirichlet boundary
conditions, have been extensively investigated in the literature due to
their wide application in biology, see e.g. \cite{Henry1981,PierreSuzuki2019,Rothe1984,Suzuki2015}. All
the previous mentioned works concern either weak solutions in all space
dimensions or classical solutions in small (one and two) space dimensions.
Until recently, Fellner \textit{et al. }\cite{FellnerMorganTang2020,FellnerMorganTang2021}
dealt with the case of skew-symmetric Lotka-Volterra systems (i.e. $A+A^{T}=0$%
) with $m-$equations and  $2\leq m\in 
\mathbb{N}
$, their results imply the existence of classical solution in all space
dimensions. It is worth noting that their results cannot guarantee
the global existence of solutions to the system $\left( \text{\ref{3SHOIs_LV}%
}\right) $ with nonhomogeneous Robin boundary conditions, even at least for $%
\gamma _{i}=1$, for $i=1.2.3$. The outcomes of our study ensure the global
existence of classical solutions of a class of $\left( \text{\ref{3SHOIs_LV}}\right) $,
where these systems hold for a wide range of boundary conditions $\left( 
\text{\ref{HNBCs_EX3}}\right) $-$\left( \text{\ref{MixNHRBCsHDBCs_EX3}}%
\right) $, and $\gamma _{i}\geq 1$ for $i=1.2.3$. In order to simplify the
notations, we define%
\begin{equation*}
\mathcal{SK}_{3\times 3}=\left\{ A=(a_{ij})_{1\leq i,j\leq 3}\in 
\mathbb{R}
^{3\times 3}\ :\ A+A^{T}\leq 0\right\} ,
\end{equation*}

\begin{equation*}
\mathcal{SK}_{3\times 3}^{+}=\left\{ A\in \mathcal{SK}_{3\times 3}:\text{ }%
a_{ij}\geq 0\text{ with }i<j\text{, for }i,j=1,2,3\right\} ,
\end{equation*}

and

\begin{equation*}
\mathcal{SK}_{3\times 3}^{-}=\left\{ A\in \mathcal{SK}_{3\times
3}:a_{ij}\leq 0\text{ with }i<j\text{, for }i,j=1,2,3\right\} .
\end{equation*}

By employing our aforementioned approach, we can deduce the following result.

\begin{theorem}[HOIs 3-species Lotka-Voltera model]
Let $N\in\mathbb{N}$, and $\Omega \subset \mathbb{R}^{N}$ is a bounded domain with smooth boundary $\partial \Omega $, where $\tau
_{i}\in\mathbb{R}$, and $\gamma _{i}\geq 1$ for $i=1,2,3$. Additionally, let $A\in \mathcal{SK%
}_{3\times 3}^{-}$ or $A\in \mathcal{SK}_{3\times 3}^{+}$. Then for any
non-negative initial conditions $u_{0},v_{0},w_{0}\in L^{\infty }\left(
\Omega \right) $, and when one of the boundary conditions $\left( \text{\ref%
{HNBCs_EX3}}\right) $-$\left( \text{\ref{MixNHRBCsHDBCs_EX3}}\right) $ is
imposed, there exists a unique nonnegative global classical solution to $%
\left( \text{\ref{3SHOIs_LV}}\right) $.
Moreover, if $\tau_i\leq 0$ for all $i=1,2,3$ and in \eqref{NHRBCs_EX3} or \eqref{MixNHRBCsHDBCs_EX3} we have $\beta_i\equiv 0$ for $i=1,2,3$, then the solution is bounded uniformly in time, i.e. 
\begin{equation}\label{uvw_LinftyUniformUpperBound_GBCsEX3}
		\sup_{t\geq 0}\|u(.,t)\|_{L^{\infty}(\Omega)},\quad\sup_{t\geq 0}\|v(.,t)\|_{L^{\infty}(\Omega)},\quad \sup_{t\geq 0}\|w(.,t)\|_{L^{\infty}(\Omega)} < +\infty.
\end{equation}
\end{theorem}

\begin{proof}
Set 
\begin{equation}  \label{RFs_EX3proof}
\left\{ 
\begin{array}{c}
\widetilde{f}(S):=u^{\gamma _{1}}\left(
a_{11}u^{\gamma _{1}}+a_{12}v^{\gamma _{2}}+a_{13}w^{\gamma _{3}}\right) ,
\\ 
\widetilde{g}(S):=v^{\gamma _{2}}\left(
a_{21}u^{\gamma _{1}}+a_{22}v^{\gamma _{2}}+a_{23}w^{\gamma _{3}}\right) ,
\\ 
\widetilde{h}(S):=w^{\gamma _{3}}\left(
a_{31}u^{\gamma _{1}}+a_{32}v^{\gamma _{2}}+a_{33}w^{\gamma _{3}}\right) ,%
\end{array}%
\right.
\end{equation}

and%
\begin{equation*}
\left\{ 
\begin{array}{c}
\widetilde{\mathcal{C}}_{1}(S):=\lambda _{1}\widetilde{f}(S)+%
\widetilde{g}(S)+\widetilde{h}(S), \\ 
\widetilde{\mathcal{C}}_{2}(S):=\lambda _{2}\widetilde{f}(S)+\lambda
_{2}\widetilde{g}(S)+\widetilde{h}(S), \\ 
\widetilde{\mathcal{C}}_{3}(S):=\lambda _{1}\lambda _{2}f(S)+\lambda
_{2}g(S)+h(S).
\end{array}%
\right.
\end{equation*}%
The verification of the assumptions \ref{a1}\ and \ref{a4} is
straightforward. Since $A\in \mathcal{SK}_{3\times 3}^{-}\subset \mathcal{SK}
$ or $A\in \mathcal{SK}_{3\times 3}^{+}\subset \mathcal{SK}$, thus the
assumption \ref{a2} is easily obtained. Now, by substituting $\left( \text{%
\ref{RFs_EX3proof}}\right) $ into the left-hand side of the inequalities $%
\left( \text{\ref{EqCond_a3}}\right) $, we get%
\begin{equation*}
\left\{ 
\begin{array}{c}
\widetilde{\mathcal{C}}_{1}(S)\leq \left( \allowbreak \lambda
_{1}a_{12}+a_{21}\right) u^{\gamma _{1}}v^{\gamma _{2}}+\left( \allowbreak
\lambda _{1}a_{13}+a_{31}\right) u^{\gamma _{1}}w^{\gamma _{3}}, \\ 
\widetilde{\mathcal{C}}_{2}(S)\leq \left( \lambda _{2}a_{13}+\allowbreak
a_{31}\right) u^{\gamma _{1}}w^{\gamma _{3}}+\left( \lambda
_{2}a_{23}+a_{32}\right) v^{\gamma _{2}}w^{\gamma _{3}}, \\ 
\widetilde{\mathcal{C}}_{3}(S)\leq \left( \lambda _{1}\lambda
_{2}a_{13}+a_{31}\right) u^{\gamma _{1}}w^{\gamma _{3}}+\left( \lambda
_{1}a_{12}+a_{21}\right) u^{\gamma _{1}}v^{\gamma _{2}}\lambda _{2}+\left(
\lambda _{2}a_{23}+a_{32}\right) v^{\gamma _{2}}w^{\gamma _{3}}.%
\end{array}%
\right.
\end{equation*}%
Therefore, in order for the following inequalities to be achieved (according to the non-negativity of each component of \( S \))%
\begin{equation*}
\left\{ 
\begin{array}{c}
\widetilde{\mathcal{C}}_{1}(S)\leq 0, \\ 
\widetilde{\mathcal{C}}_{2}(S)\leq 0, \\ 
\widetilde{\mathcal{C}}_{3}(S)\leq 0,
\end{array}%
\right.
\end{equation*}%
it is enough to verify that 
\begin{equation}
\left\{ \allowbreak 
\begin{array}{c}
\allowbreak \lambda _{1}a_{12}+a_{21}\leq 0, \\ 
\allowbreak \lambda _{1}a_{13}+a_{31}\leq 0, \\ 
\lambda _{2}a_{13}+\allowbreak a_{31}\leq 0, \\ 
\lambda _{2}a_{23}+a_{32}\leq 0, \\ 
\lambda _{1}\lambda _{2}a_{13}+a_{31}\leq 0.%
\end{array}%
\right.  \label{AcondEX3}
\end{equation}%
So we distinguish two cases:

\begin{itemize}
\item When $A\in \mathcal{SK}_{3\times 3}^{-}:$ It hold that $a_{ij}\leq 0$
with $i<j$, and $a_{ij}+a_{ji}\leq 0$, for $i,j=1,2,3$. Consequently,
choosing any values $\lambda _{1},\lambda _{2}>1$ is sufficient to ensure
that \eqref{AcondEX3} is satisfied, and thus the assumption \ref{a3}. Hence, by applying
Theorem \ref{main_Theor3GBCs}, the desired result is established.

\item When $A\in \mathcal{SK}_{3\times 3}^{+}:$ It hold that $a_{ij}\geq 0$
with $i<j$, and $a_{ij}+a_{ji}\leq 0$, for $i,j=1,2,3$. Consequently,
choosing any values $\lambda _{1},\lambda _{2}<1$ is sufficient to ensure
that \eqref{AcondEX3} is satisfied, and thus the assumption \ref{a3}. Hence, By invoking Theorem \ref{main_Theor4GBCs}, the desired result is established.
\end{itemize}
This completes the proof.
\end{proof}
\begin{remark}
    If the diffusion coefficients differ from each other, the question of whether system \eqref{3SHOIs_LV} possesses global classical solutions(with or without uniform-in-time bounds) or undergoes blow-up remains open, in the following cases:
    \begin{itemize}
    \item When $\gamma_{i} = 1$ for $i=1,2,3$, and $A \in \mathcal{SK}_{3\times 3} \backslash \left( \mathcal{SK}_{3\times 3}^{-} \cup \mathcal{SK}_{3\times 3}^{+} \right)$, under the boundary condition \eqref{NHRBCs_EX3} or \eqref{MixNHRBCsHDBCs_EX3}.
    \item When $\gamma_{i} > 1$ for $i=1,2,3$, and $A \in \mathcal{SK}_{3\times 3} \backslash \left( \mathcal{SK}_{3\times 3}^{-} \cup \mathcal{SK}_{3\times 3}^{+} \right)$, under one of the boundary conditions \eqref{HNBCs_EX3}–\eqref{MixNHRBCsHDBCs_EX3}.
    \item When $\gamma_{i} \geq 1$ for $i=1,2,3$, and $A \in \mathcal{SK}_{3\times 3}$, under nonlinear boundary conditions.
    \end{itemize}
\end{remark}

\end{document}